\documentclass[12pt,a4paper,  reqno, english]{amsart}
\usepackage[utf8]{inputenc}
\usepackage[margin=1in]{geometry} % Required for inputting international characters
\usepackage{amsthm}
\usepackage{amsmath}
\usepackage{amsfonts}
\usepackage{amssymb}
\usepackage{chngcntr}
\usepackage{url}
\usepackage{hyperref}
\usepackage[makeroom]{cancel}
\usepackage{enumerate}
\usepackage{mathrsfs}
\newtheorem{Theorem}{Theorem}

\newtheorem{Corollary}{Corollary}[section]
\newtheorem{Lemma}{Lemma}[section]

\newtheorem{Case}{Case}

\numberwithin{equation}{section}
\title{A tight structure theorem for sumsets}
\title{A tight structure theorem for sumsets}
\author{Andrew Granville}
\thanks{A.G. was funded by the European Research Council grant agreement no 670239, and by the Natural
Sciences and Engineering Research Council of Canada (NSERC) under the Canada Research Chairs
program.} 
\address{AG: D\'{e}partement de math\'{e}matiques et de statistique, Universit\'{e} de Montr\'{e}al, CP
6128 succ. Centre-Ville, Montr\'{e}al, QC H3C 3J7, Canada.}
\email{andrew@dms.umontreal.ca}
  \author{Aled Walker}
\thanks{A.W. was supported by a postdoctoral research fellowship at the Centre de Recherches Mathématiques, and is also a Junior Research Fellow at Trinity College Cambridge.}
\thanks{The authors would like to thank George Shakan for helpful communications}
\address{AW: Centre de recherches mathématiques, Université de Montréal, CP 6128 succ. Centre-ville Station, Montréal, QC H3C 3J7, Canada; and Trinity College, Cambridge CB2 1TQ, England.}
\email{aw530@cam.ac.uk}
\begin{document}
\begin{abstract}
Let $A = \{0 = a_0 < a_1 < \cdots < a_{\ell + 1} = b\}$ be a finite set of non-negative integers. We prove that the sumset $NA$ has a certain easily-described structure, provided that $N \geqslant b-\ell$, as recently conjectured in \cite{GS20}. We also classify those sets $A$ for which this bound cannot be improved. 
\end{abstract}
\maketitle

%%%%%%%%%%%%%%%%%%%%%%%
\section{Introduction}
What are the possible postage costs that can be made up from an unlimited supply of $3$ cent
and $5$ cent stamps? One cannot obtain $1c$, $2c$, $4c$, or $7c$ and it is a fun challenge  to show that one can obtain $n$ cents for every other positive integer $n$.
In \emph{the Frobenius postage stamp problem}, one asks the same question given an unlimited supply of $a$ cent and $b$
cent stamps, with $\gcd(a, b) = 1$.

The situation becomes more complicated if one may use at most $N$ stamps. One can show that one can cover every integer amount up to $5N$ cents using at most $N$ $3$  
and $5$ cent stamps, other than 
$1,2,4$ and $7$, as well as $5N-3$ and $5N-1$.

In the language of additive combinatorics, for a given finite set of integers $A$ we wish to understand the structure of the sumset $NA$, where 
\[ 
NA : = \{ a_1 + \cdots + a_N: a_1,\dots,a_N \in A\},
\] 
where the summands are not necessarily distinct. For simplicity we may assume without loss
of generality that the smallest element of $A$ is $0$, and that the greatest common divisor of its elements is $1$.\footnote{Since if $A = g\cdot B + \tau$ then $NA= g\cdot NB + N\tau$, where $g \cdot B = \{gb: b \in B\}$.} Since $0 \in A$ we have $A \subset 2A \subset \cdots \subset NA$, and so 
\[ 
\mathcal{P}(A) := \bigcup \limits_{N=1}^{\infty}NA
\] 
is the set of all integers that are expressible as a finite sum of (not necessarily distinct) elements of $A$. Similarly, we define the exceptional set \[ \mathcal{E}(A) = \{ n \geqslant 1 : n \notin \mathcal{P}(A)\}.\] In the setting of the original postage stamp problem, in this notation we have 
\[ \mathcal{E}(\{0,3,5\}) = \{1,2,4,7\}. \] 

Let $b$ denote the largest element of $A$ so that $\{ 0,b\} \subset A  \subset  \{0,1,\dots, b\}$ which implies that
 $NA \subset \{0,1,\dots, bN\} \setminus \mathcal{E}(A)$. However, in the $A=\{ 0,3,5\}$ example  there are  exceptions other than $\mathcal{E}(A)$: Indeed, if $n \in NA$ then $bN - n \in N(b-A)$,  where $b-A := \{ b-a: a \in A\}$. Therefore $NA \cap (bN - \mathcal{E}(b-A)) = \emptyset$, and thus \[ NA \subset \{0,1,\dots,bN\} \setminus ( \mathcal{E}(A) \cup (bN - \mathcal{E}(b-A))).\] 
Does equality hold in this expression? 
In our example $b-A=\{0,2,5\}$ and $ \mathcal{E}(\{0,3,5\}) = \{1,3\}$ which explains the result above.
It was shown in \cite{GS20} that equality indeed holds for all $N\geq 1$ for all three element sets $A= \{0<a<b\}$ where $(a,b)=1$. If $A=\{ 0,1,b-1,b\}$ then equality does not hold for any $N\leq b-3$ since $\mathcal{E}(A) = \mathcal{E}(b-A) = \emptyset$ and $b-2\not\in NA$ for such $N$. 

Our main result gives an improved bound for the smallest $N_0$, such that we get equality above; that is, \eqref{eq structure theorem} for all $N\geq N_0$. This improved bound is ``best possible'' in several situations.

 \begin{Theorem} [Main theorem]
 \label{thm 1} 
Let $A = \{0 = a_0 < a_1 < \cdots < a_{\ell+1} = b\}$ be a finite set of integers with $\gcd(a_1,\dots,a_{\ell+1}) = 1$ and $\ell \geqslant 1$. If $N \geqslant b-\ell$ then 
 \begin{equation}
 \label{eq structure theorem}
 NA = \{0,1,\dots,bN\} \setminus ( \mathcal{E}(A) \cup (bN - \mathcal{E}(b-A))).
 \end{equation}
 \end{Theorem}
 \smallskip
 
A statement like Theorem \ref{thm 1} was first proved by Nathanson \cite{Na72}, but with the weaker bound $N \geqslant b^2 (\ell + 1)$. The bound was in turn improved to $N \geqslant \sum_{ a \in A, \, a \neq 0} (a-1)$ in \cite{WCC11}, and then to $N \geqslant 2 \lfloor \frac{b}{2} \rfloor$ in \cite{GS20}, where our bound $N\geq b-\ell$ was conjectured. \\

The bound ``$N \geqslant b-\ell$'' in Theorem \ref{thm 1} is tight, in that there are examples of sets $A$ for which   (\ref{eq structure theorem}) does not hold when $N = b-\ell - 1$. In particular there are the following families:
\begin{itemize}
\item $A = \{0,1,\dots,b\} \setminus \{a\}$ for some $a$ in the range $2 \leqslant a \leqslant b-2$. Here $ b-\ell - 1 = 1$ and $\mathcal{E}(A) = \mathcal{E}(b-A) = \emptyset$, but $a \not\in  A$, in contradiction to \eqref{eq structure theorem}.
\item $A = \{0,1,a+1,\dots,b-1,b\}$, for some $a$ in the range $2 \leqslant a \leqslant b-2$. Here $ b - \ell - 1 = a-1$ and $\mathcal{E}(A) = \mathcal{E}(b-A)= \emptyset$, but $a \not\in  (a-1)A$, contradicting \eqref{eq structure theorem}.
\end{itemize}
The previous bounds of \cite{WCC11} and \cite{GS20} were also tight for certain special values of $\ell$ and $b$, but our Theorem \ref{thm 1} is the first such bound for which tight examples exist for all $b \geqslant 4$ and for all $\ell$ in the range $2 \leqslant \ell \leqslant b-2$.

Moreover, it turns out that the families listed above are the only obstructions to improving Theorem \ref{thm 1}:

\begin{Theorem}
\label{thm 2}
Let $\ell \geqslant 1$ and $A = \{0 = a_0 < a_1 < \cdots < a_{\ell+1} = b\}$ be a finite set of integers with $\gcd(a_1,\dots,a_{\ell+1}) = 1$. If $N \geqslant \max(1,b-\ell - 1)$ then 
 \begin{equation*}
 NA = \{0,1,\dots,bN\} \setminus ( \mathcal{E}(A) \cup (bN - \mathcal{E}(b-A))),
 \end{equation*}
unless either $A$ or $b-A$ is a set in one of the two families listed above. 
 \end{Theorem}
 
Our goal in proving Theorems \ref{thm 1} and \ref{thm 2} was to establish tight bounds in the venerable Frobenius postage stamp problem. These bounds can now be applied to what we hope is a cornucopia of questions in additive combinatorics (for example, Corollary 1.8, Lemma 5.1, and Lemma 5.2 in  \cite{ILMZ12}) where explicit tight bounds are needed.

Our methods also show that, if $b \geqslant 9$ and $\ell \geqslant 5$, then \eqref{eq structure theorem} holds for all $N \geqslant \max(1,b-\ell - 2)$ unless $A$ or $b-A$ belong to one of the two families listed above or one of the following  new families:
\begin{itemize}
\item $A = \{0,1,b\} \cup (\{a+1,\dots,b-1\} \setminus \{d\})$ for some $a$ in the range $2\leqslant a \leqslant b-2$ and some $d$ in the range $a+2\leqslant d \leqslant b-1$, where $a\not\in (a-1)A$;
\item $A = \{ 0,1,\dots, b\}\setminus \{ a,c\}$ for some $2 \leqslant a,c \leqslant b-2$, where $a\not\in A$;
\item $A = \{ 0,1,2,6,\dots,b\}$, where $5\not\in 2A$;
\item $A = \{ 0,1,3,6,\dots,b\}$, where $5\not\in 2A$. 
\end{itemize}
Indeed, our proofs are sufficiently flexible that one can go on and prove that \eqref{eq structure theorem} holds for all $N \geqslant \max(1,b-\ell - \Delta)$, for ever larger values of $\Delta$, except in some explicit finite set of families of sets $A$, though  the number of cases seems to grow prohibitively with $\Delta$.  \\

The final parts of the proofs of Theorems \ref{thm 1} and \ref{thm 2} come in Section \ref{section main proofs}. These will rely on a number of auxiliary lemmas and use some terminology from \cite{GS20}, all of which we will introduce in the preceding sections. There are a few families of examples, like $A = \{0,h,b-h,b\}$ with $(h,b)=1$, for which our general arguments for Theorems \ref{thm 1} and \ref{thm 2} fail, and for these examples
we verify the theorems  explicitly in Appendix \ref{section checking examples}.\\

 %%%%%%%%%%%%%%%%%%%%%%%
 \section{Placing elements in $NA$}
\label{section main lemma}
Throughout we fix a set $A \subset \mathbb{Z}$ with minimum element $0$ and maximum element $b$, where
$A\setminus \{ 0,b\}$ has $\ell$ elements, and $\text{gcd}(a: a\in A)=1$.  Let $B$ be the reduction of $A \pmod b$ so that $\vert B\vert = \ell + 1$, and its elements can be represented by $A \setminus \{b\}$.

For  $a$ in the range $1 \leqslant a \leqslant b-1$ we write 
\[ 
n_{a,A}: = \min\{n \geqslant 1: n \in \mathcal{P}(A), \, n \equiv a \pmod b\}
\] 
and 
\[
N_{a,A}: = \min\{N \geqslant 1: n_{a,A} \in NA\}, \text{ with }   N_{A}^*: = \max\limits_{ 1 \leqslant a \leqslant b-1} N_{a,A}.
\]
We always have $N_{a,A}\leq b-1$ for if not we write $n_{a,A}=a_1+\dots+a_N$ with each $a_i\in A$ and $N=N_{a,A}$. Then at least two of $b+1$ subsums
\[
0,a_1,a_1+a_2,\dots , a_1+\dots+a_b
\]
must be congruent mod $b$, say $a_1+\dots+a_i\equiv a_1+\dots+a_j \pmod b$ with $i<j$, and then
\[
a_1+\dots+a_i+a_{j+1}\dots+a_N \equiv a_1+\dots+a_N=n_{a,A}\equiv a \pmod b,
\]
contradicting the minimality of $n_{a,A}$.

 It was observed in \cite{GS20} that 
\begin{equation}
\label{eq structure of EA}
 \mathcal{E}(A) = \bigcup\limits_{a=1}^{b-1} \{n \geqslant 1: n < n_{a,A}, \, n \equiv a \pmod b\}
 \end{equation} 
 so that $  \{0,1,\dots,bN\} \setminus (\mathcal{E}(A) \cup (bN - \mathcal{E}(b-A))) $ equals
 \[ 
  \bigcup \limits_{a=1}^{b-1} \{n: n_{a,A} \leqslant n \leqslant bN - n_{b-a,b-A}, \, n \equiv a \pmod b\}.
  \] 
  Therefore the equality \eqref{eq structure theorem} holds if and only if the arithmetic progressions
\begin{equation}
\label{eq aritmetic progressions}
\{n: n_{a,A} \leqslant n \leqslant bN - n_{b-a,b-A}, \, n \equiv a \pmod b\}, \qquad 1 \leqslant a \leqslant b-1
\end{equation}
are contained in $NA$. Our first lemma shows that, under certain conditions on the sumset of $B$, elements of the arithmetic progressions in \eqref{eq aritmetic progressions} do belong to $NA$.

 \begin{Lemma} \label{L1}
Let $a$ be in the range $1 \leqslant a \leqslant b-1$, let $k \geqslant 1$ and suppose that $ |kB| \geqslant b - N_{a,A}$. Then $n_{a,A}+(k-1)b\in NA$ whenever $N\geqslant 2k+b-|kB| - 1$.
\end{Lemma}
 
 \begin{proof} Suppose that 
 \[
 n_{a,A} = a_1+\dots+a_L \text{ where } L:=N_{a,A},
 \] for some $a_i \in A$ not necessarily distinct. 
 Consider then the set $\mathcal M$ of subsums
 \[
 a_1+\dots+a_M, a_1+\dots+a_{M+1},\dots, a_1+\dots+a_{L}
 \]
 where $M=N_{a,A}-(b-|kB|)$, where if $M = 0$ we consider the first (empty) sum to be equal to $0$. 
 
We make several observations. First, since $b \geqslant \vert kB\vert \geqslant b-N_{a,A}$ we have $N_{a,A} \geqslant M \geqslant 0$, so the construction of $\mathcal{M}$ is valid. Second, we observe that the members of $\mathcal{M}$ are distinct mod $b$ by the definition of $n_{a,A}$. To justify this second part, we note that if two members of $\mathcal{M}$ were the same modulo $b$ then there would be a subsum of $a_1 + \cdots + a_L$ congruent to $0$ mod $b$, say $\sum_{s \in S} a_s$. Furthermore we know that $a_s \geqslant 1$ for all $s$, by the minimality of $N_{a,A}$. But then $n: = n_{a,A} - \sum_{s \in S} a_s$ satisfies $n < n_{a,A}$, $n \equiv a$ mod $b$, and $n \in \mathcal{P}(A)$, which contradicts the minimality of $n_{a,A}$. 
 
 Now $|\mathcal M|+|kB|\geqslant b+ 1$ and the elements of  $\mathcal M$ are distinct mod $b$.  Therefore, by the pigeonhole principle, there exists an integer $m\in [M,L]$ for which 
 \[
 a_1+\dots+a_m\in kB \mod b;
 \]
that is, there exists an integer $i$ and $b_1,\dots, b_k\in A\setminus \{b\}$ for which 
\[    a_1+\dots+a_m+ib = b_1+\cdots + b_k .\]
We may extract some bounds for $i$. Indeed, note that $ib \leqslant a_1+\dots+a_m+ib = b_1+\cdots + b_k<kb$ and so $i \leqslant k-1$. Also
\[
n_{a,A}+ib= (a_1+\dots+a_m+ib) +(a_{m+1}+\dots+a_L)  = (b_1+\cdots + b_k)+(a_{m+1}+\dots+a_L) \in \mathcal P(A)
\]
and so $i\geqslant 0$ by the minimality of $n_{a,A}$.

Therefore
\begin{align*}
n_{a,A}+(k-1)b&=(b_1+\cdots + b_k)+(a_{m+1}+\dots+a_L)  +(k-1-i)b\\
&\in k(A\setminus \{b\}) + (N_{a,A}-m)A+(k-1-i)b\\
&\subset (k+(b-|kB|)+k - 1)A\subset NA,
\end{align*}
since $i\geqslant 0$ and $N_{a,A}-m\leqslant N_{a,A}-M=b-|kB|$.
\end{proof}

We will combine this lemma with some lower bounds on the growth of the sumset $\vert kB\vert$. Our main tool is Kneser's theorem \cite[Theorem 5.5]{TV06}, which states that if $U,V$ are subsets of a finite abelian group $G$ then
 \[
 |U+V|\geqslant |U+H|+ |V+H| -|H|
 \]
 where $H=H(U+V)$ is the \emph{stabilizer} of $U+V$, defined in general by \[H(W):=\{ g\in G: g+W=W\}.\] 
One notes in particular that $V+H$ is a union of cosets of $H$, so its size is a multiple of $|H|$. Therefore if $0\in V$ but $V\not\subset H$ then $|V+H|-|H|\geqslant |H|$.
\begin{Lemma}
\label{Lemma Kneser induction}
Assume that $\ell \geqslant 2$. For all $k \geqslant 2$, $\vert kB\vert \geqslant \min(b, \vert (k-1)B\vert +2)$.
\end{Lemma}
\begin{proof}
By Kneser's theorem we have
 \begin{align*}
 |k B|&\geqslant |(k-1)B+H(k B)|+|B+H(k B)|-|H(k B)|\\
 &\geqslant \vert (k-1)B\vert  +|B+H(k B)|-|H(k B)|. 
\end{align*}
 If $H(k B)=\mathbb Z/b\mathbb Z$ then $\vert k B\vert =b$ and we are done, so we may assume that $H(kB)$ is a proper subgroup of $\mathbb{Z}/b\mathbb{Z}$. Since $B$ generates all of $\mathbb Z/b\mathbb Z$ we see that $B\not\subset H(k B)$ 
 and so $|B+H(k B)|-|H(k B)|\geqslant |H(k B)|$.  Therefore if $H(k B)\ne \{ 0\}$ then $|k B|\geqslant |(k -1)B|+2$. If on the other hand we have $H(k B)= \{ 0\}$ then \[|k B|\geqslant |(k -1)B|+|B|-1=|(k-1)B|+\ell\geqslant \vert (k-1)B\vert + 2\] since $\ell \geqslant 2$. 
 \end{proof}

We make a deduction, phrased in a suitably general way so as to apply in the setting of both Theorems \ref{thm 1} and \ref{thm 2}. 

\begin{Corollary} \label{corr1}
Assume that $\ell \geqslant 2$, and let $N=b-\ell-\Delta$ for some $\Delta \geqslant 0$. Let $K$ be the smallest integer such that $K\geqslant 2$ and $\vert KB\vert \geqslant \min(b,2K + \ell + \Delta - 1)$, and assume that $N \geqslant N_A^* + K - 2$. Then $n \in NA$ for all $n \leqslant bN/2$ with $n \notin \mathcal{E}(A)$. 
\end{Corollary}

 \begin{proof} We will show that $n_{a,A} + kb \in NA$ for all $k< N/2$
and all $a$ in the range $1 \leqslant a \leqslant b-1$, which implies the result, by  \eqref{eq structure of EA}. 

Note that $N \geqslant N_A^* \geqslant N_{a,A}$. Therefore if $0 \leqslant k \leqslant N - N_{a,A}$ we have $n_{a,A} + kb \in N_{a,A} A + (N-N_{a,A})A = NA$, so without loss of generality we may assume that $k \geqslant N-N_{a,A} + 1$, so that $k+1\geqslant N-N_{a,A} + 2 \geqslant N-N_A^* + 2 \geqslant K$. 

From Lemma \ref{Lemma Kneser induction} and induction, this means that $\vert (k+1)B\vert \geqslant \min(b,2(k+1) + \ell + \Delta - 1)$. 
Our goal is to apply Lemma \ref{L1} with $k$ replaced by $k+1$ so we need to verify its hypotheses:
\begin{itemize}
\item If $\vert (k+1)B\vert = b$ then $\vert (k+1)B\vert \geqslant b-N_{a,A}\geqslant N-N_{a,A}$ trivially, and
\[ 2(k+1) + b - \vert (k+1)B\vert - 1 = 2k + 1\leqslant N\] since $k<\frac N2$;
\item Otherwise  $|(k+1)B|\geqslant 2(k+1)+\ell+\Delta - 1$, and so we have both
 \[2(k+1)+b-|(k+1)B|-1\leqslant b-\ell-\Delta=N\] and 
 \[
 |(k+1)B|\geqslant 2(N-N_{a,A}+2)+\ell - 1+\Delta\geqslant N+\ell+\Delta-N_{a,A}+3=b-N_{a,A}+3
 \]
 as $N\geqslant N_{a,A}$. 
 \end{itemize}
Therefore Lemma \ref{L1} implies that $n_{a,A} + (k+1 - 1)b = n_{a,A} + kb \in NA$, as desired. 
 \end{proof}

 %%%%%%%%%%%%%%%%%%%%%%%
 \section{Bounds on $\vert 2B\vert$ and $N_{A}^*$}
 
In order to use Corollary \ref{corr1}, two further bounds will be useful: a lower bound on $\vert 2B\vert$ and an upper bound on $N_{A}^*$.  We will achieve both of these objectives in this section (bar a few special cases which we will deal with separately).

  \begin{Lemma} \label{L.Kneser} 
  Suppose that $B$ is a subset of $\mathbb Z/b\mathbb Z$ which contains $0$, generates all of $\mathbb Z/b\mathbb Z$, and has $\ell\geqslant 2$ non-zero elements. Then $|2B|\geqslant \min( b,\ell+3)$. Moreover $|2B|\geqslant \min( b,\ell+4)$, except in the following  families of examples (where   $A=B\cup\{ b\}$ for later convenience):
\begin{itemize}
\item $A = \{0 < h< 2h < b\}$ with $(h,b) = 1$; 
\item $A = \{0 < 2h - b < h < b\}$ with $(h,b) = 1$;
\item $A = \{0<h<b-h < b\}$ with $(h,b) = 1$;
\item $A = \{0 < h,\frac{b}{2} < b\}$ with $(h,\frac{b}{2}) =1$;
\item $A = \{0 < h < h + \frac{b}{2} < b\}$ with $(h,\frac{b}{2}) =1$;
\item $A = \{0 < h < b/2 < h + \frac{b}{2} < b\}$ with $(h,\frac{b}{2}) =1$.
\end{itemize}
\end{Lemma}

\begin{proof}  We aim to prove that $|2B|\geqslant \min( b,\ell+3 + \Delta)$ for $\Delta = 0$ or $1$.
By Kneser's theorem we have 
\[ 
|2B|\geqslant 2|B+H|-|H|,
\] 
where $H=H(2B)$.
 If $\vert H\vert = b$ then $\vert 2B\vert = b$ and we are done. 
 
 If $H= \{ 0\}$ then we derive $\vert 2B\vert \geqslant 2\ell+1\geqslant \ell+3+\Delta$, provided 
 $\ell \geqslant 2+\Delta$. Therefore we are done unless $\Delta=1$ and $\ell=2$ with  $\vert 2B\vert \leqslant 5$.
 In this case $B=\{ 0,h,k\}$ with $(h,k,b)=1$ and at least two of $0,h,k,2h,h+k,2k$ must be congruent mod $b$.
 One obtains the first five families of examples in the result from a case-by-case analysis (letting $k=2h,2h-b,b-h,\frac b2$ and $h+\frac b2$, respectively).
 
  Now we may assume that $2 \leqslant \vert H\vert \leqslant b-1$. If  $B$ is not a union of $H$-cosets then $\vert B + H\vert \geqslant \vert B\vert + 1 = \ell + 2$. Also, since $B$ generates $\mathbb{Z}/b\mathbb{Z}$ and $\vert H\vert \neq b$ we have $B \not\subset H$, and so $|B+H|\geqslant 2|H|$. Thus
\[
|2B|\geqslant |B+H|+( |B+H|-|H|) \geqslant   \ell+2 +|H| \geqslant\ell+4.
\] 
 
Finally assume that $B$ is the union of $r$ $H$-cosets with $r\geqslant 2$, so that  
\[
|2B|\geqslant (2r-1)|H| = (2-\tfrac 1r)|B|= (2-\tfrac 1r)(\ell+1).
\]
  This is at least $\ell+3 + \Delta$  unless $\ell < \tfrac{r}{r-1} (1 + \frac{1}{r} + \Delta)$, where $r,\ell\geq 2$ and $r$ is a proper divisor of $\ell+1$, and so $\ell\ne 2$ or $4$. 
  For $\Delta=0$ this implies   $\ell<3$, which is impossible. If $\Delta=1$ the inequality implies    
  $\ell < 2+\tfrac{3}{r-1}\leq 5$, so that the only possibility is $\ell=3$ and $r=2$, so that $|H|=2$. Therefore
  $b$ is even,  $H=\{ 0,\tfrac b2\}$ and we obtain the sixth family of examples.
   \end{proof}

We now present some bounds on $N_{A}^*$. 

  \begin{Lemma} \label{L2} Suppose that $2 \leqslant \ell \leqslant b-2$. Then we have $N_{A}^* \leqslant b-\ell - 1$, except when:
\begin{itemize}
\item $A=\{ 0,1,\dots, b\} \setminus \{ a\}$ for some $a$ in the range $1 \leqslant a \leqslant b-1$, in which case $N_{a,A}=2$ and $N_A^* = 2 = b-\ell$; or when
\item $A=\{ 0,1,a+1,\dots,b-1,b\} $ for some $a$ in the range $2 \leqslant a \leqslant b-2$, in which case $n_{a,A}=a\times 1,\ N_{a,A}=a=b-\ell$ and $N_{A}^* = b-\ell$.
\end{itemize} 
  \end{Lemma}

\begin{proof} 
Choose some $a$ in the range $1 \leqslant a \leqslant b-1$, and let $n_{a,A}=a_1+\dots+a_{L}$ where $L:=N_{a,A}$, with each $a_i \in A$. All subsums are non-zero mod $b$, as both $n_{a,A}$ and $L$ are minimal (see the proof of Lemma \ref{L1} for a longer explanation of this fact). Furthermore a subsum with more than one element cannot be congruent mod $b$ to an element of $B$ else we can replace that subsum by the single element, contradicting the minimality of $L$. Hence the residue classes mod $b$ of 
\begin{equation}
\label{eq displayed values}
a_1+a_2,\dots,a_1+\dots+a_{L}
\end{equation}
are all distinct  and do not belong to $B \mod b$.
This yields $L-1$ distinct residue classes of $(\mathbb Z/b\mathbb Z)\setminus B$, and so
$N_{a,A}\leqslant b-\ell$. Therefore either $N_A^*\leqslant b-\ell - 1$, in which case we are done, or 
we are in a case where $N_{a,A}= b-\ell$.

If $N_{a,A} = b- \ell$ then the displayed values \eqref{eq displayed values} yield all the residue classes of $(\mathbb Z/b\mathbb Z)\setminus B$.  This is also true if we list the $a_i$ in a different order, so we can swap $a_2$ and $a_3$ and find that $ a_1+a_2\equiv a_1+a_3 \pmod b$ (as this element is the only difference between the two lists). Thus $a_2=a_3$. But this is true for any ordering of the $a_i$'s, so 
$n_{a,A}$ is given by $L=N_{a,A}$ copies of some $h\in A$. Therefore 
\begin{equation}
\label{eq structure of A}
A=\{0,1,\dots,b\}\setminus \{ (2h)_b,(3h)_b,\dots,(Lh)_b\},
\end{equation}
where $(t)_b$ denotes the least positive residue of $t \pmod b$, and $a \equiv Lh \pmod b$.

We split into two cases according to the value of the greatest common divisor $(h,b)$. If   $(h,b) > 1$ then  $1 \in A$ (since $1 \not\equiv 0$ mod $(h,b)$). Therefore $n_{a,A} = a=L h$. Moreover $a - 1 \in A$ (since $a-1\equiv -1 \not\equiv 0$ mod $(h,b)$) and   $a = (a-1) + 1$ implies that $N_{a,A} \leqslant 2$. But $N_{a,A}= b-\ell\geq 2$, and so   $\ell = b-2$. Hence $A = \{0,1,\dots,b\}\setminus \{a\}$. 

We may now assume that $(h,b) = 1$. Then $n_{a,A}=Lh$, so that \[a\equiv Lh \equiv ((L+j)h)_b+((b-j)h)_b\] for $1\leq j \leq \ell-1$. Since $L = b - \ell$ we have both $L+1 \leqslant L+j \leqslant b-1$ and $L+1 \leqslant b - j \leqslant b-1$.  Therefore by \eqref{eq structure of A} we have $((L+j)h)_b \in A$ and $((b-j)h)_b \in A$ for each $j$. Thus \[((L+j)h)_b+((b-j)h)_b\in \mathcal P(A) \text{ and }  n_{a,A}\leq ((L+j)h)_b+((b-j)h)_b <b+b=2b.\] Therefore $((L+j)h)_b+((b-j)h)_b=Lh $ or $Lh+b$. If $((L+j)h)_b+((b-j)h)_b=Lh =n_{a,A}$ then
$ N_{a,A}\leqslant 2$. Since $N_{a,A}= b-\ell\geq 2$ we conclude that $N_{a,A}=2$ and $\ell = b-2$, and so
$A=\{ 0,1,\dots, b\} \setminus \{ a\}$ again. 

Otherwise $Lh<b$ and $((L+j)h)_b+(-jh)_b =Lh+b$ for all $j$ in the range $1 \leqslant j\leqslant \ell - 1$. This means that for all such $j$ we have
\begin{equation}
\label{eq inequality chain}
b>((L+j)h)_b=Lh+(b-(-jh)_b)>Lh. 
\end{equation}
This implies that $((L+j)h)_b=(L+j)h$ for all $j$ in the range $1 \leqslant j \leqslant \ell - 1$. Indeed, suppose for contradiction that $j$ in that range is minimal such that $((L+j)h)_b < (L+j)h$. Then $((L + j-1)h)_b = (L + j-1)h < b$, by the assumption of minimality if $j \geqslant 2$, or by \eqref{eq inequality chain} if $j = 1$. So $((L+j)h)_b < h \leqslant Lh$. This is a contradiction to \eqref{eq inequality chain}.  

Choosing $j = \ell-1$ in the equation $((L+j)h)_b=(L+j)h$ we deduce that $(b-1)h<b$ and so $h=1$. Therefore $A = \{0,1,a+1,\dots,b-1,b\}$ with $a=L=b-\ell$ and so $2 \leqslant a \leqslant b-2$.
Thus we have established that $N_A^*\leqslant b-\ell - 1$ except when $A$ has one of the two special forms listed in the statement of the lemma. 
\end{proof}

%%%%%%%%%%%%%%%%%%%%%%%
\section{Proofs of Theorems \ref{thm 1} and \ref{thm 2}}
\label{section main proofs}

The result \cite[Theorem 4]{GS20} showed that \eqref{eq structure theorem} holds for all $N \geqslant 1$ when $\ell = 1$ (and so it holds for all $N \geqslant b-\ell$). Furthermore, \eqref{eq structure theorem} holds for trivial reasons if $\ell = b-1$, i.e. if $A = \{0,1,\dots,b\}$. So without loss of generality may assume $2 \leqslant \ell \leqslant b-2$ in these two proofs.

% Having laid the groundwork in the preceding sections, we are now in a position to prove our main results. The proofs are similar, but Theorem \ref{thm 2} requires the analysis of some additional special cases. 

\begin{proof}[Proof of Theorem \ref{thm 1}]
We have  $N_A^* \leqslant b-\ell$ by  Lemma \ref{L2}, and  $\vert 2B\vert \geqslant \min(b,\ell + 3)$ by Lemma \ref{L.Kneser}. Taking   $\Delta = 0$ and $K = 2$ in Corollary \ref{corr1}, we deduce that if $N =b-\ell$ then $n \in NA$ for all $n \leqslant bN/2$ with $n \notin \mathcal{E}(A)$. Applying the same argument with the set $A$ replaced by the set $b-A$ we conclude that if $N = b-\ell$ then $m \in N(b-A)$ for all $m \leqslant bN/2$ with $m \notin \mathcal{E}(b-A)$. 

So, if $1 \leqslant n < bN$ and $n \not\in (\mathcal{E}(A) \cup (bN - \mathcal{E}(b-A)))$ then either $1 \leqslant n \leqslant bN/2$, in which case $n \in NA$ by the applying the first argument to $n$, or $bN/2 \leqslant n < bN$, in which case $n \in NA$ by applying the second argument to $m = bN - n$. Since $bN \in NA$ for trivial reasons, we have established \eqref{eq structure theorem} for $N = b - \ell$. 

The result \cite[Lemma 2]{GS20}  established that if \eqref{eq structure theorem} holds for some $N_0 \geqslant N_{A}^*$ then it holds for all $N \geqslant N_0$. So \eqref{eq structure theorem} holds for all $N\geqslant b-\ell$, and Theorem \ref{thm 1} is proved.
\end{proof}

\begin{proof}[Proof of Theorem \ref{thm 2}]Following the proof of Theorem \ref{thm 1},
we have  $N_A^* \leqslant b-\ell-1$ except in the two exceptional cases of  Lemma \ref{L2}, and  $\vert 2B\vert \geqslant \min(b,\ell + 4)$  except in the six exceptional cases of  Lemma \ref{L.Kneser}. Outside these exceptional cases, the proof then follows analogously to the proof  of Theorem \ref{thm 1}, taking   $\Delta = 1$ and $K = 2$ in Corollary \ref{corr1}.

It remains to consider the exceptional cases. All of the exceptional cases in Lemma \ref{L2} are excluded in the statement of Theorem \ref{thm 2}, except for when
$A$ or $b-A$ equals $\{0,2,3,\dots,b\}$. In this instance, $\mathcal{E}(A) = \{1\}$ or $\mathcal{E}(b-A) = \{1\}$, respectively, and \eqref{eq structure theorem} manifestly holds for all $N \geqslant 1 = b- \ell - 1$.

Regarding the exceptional cases from Lemma \ref{L.Kneser}, the example $A = \{0,1,b-1,b\}$ is excluded from  Theorem \ref{thm 2}. We will prove that equation \eqref{eq structure theorem} holds for $N \geqslant  b-\ell - 1$ for all of 
 the other exceptional cases from Lemma \ref{L.Kneser} (in the five cases of Section \ref{subsection alpha equals 2 Kneser}, and in Case \ref{Case first nontrivial Kneser} of Appendix \ref{section checking examples}). This completes the proof of Theorem \ref{thm 2}. 
\end{proof}

 %%%%%%%%%%%%%%%%%%%%%%%
 \appendix
\section{A catalogue of exceptional cases}
\label{section checking examples}

 %%%%%%%%%%%%%%%%%%%%%%%

  \subsection{Resolving the five exceptional families of $A$ for which $|A|=4$ and  $|2B|\leqslant 5$} 
  \label{subsection alpha equals 2 Kneser}
  
 %  The results in this section combine to give the following:
  
%  \begin{Proposition} \label{prop2Bsmall} Suppose that  $|B|=3$ and $|2B|\leqslant 5$. Then the structure theorem holds for all $N \geqslant b-3$, except for when $A=\{ 0,1,b-1,b\}$ in which it holds fort all $N \geqslant b-2$.    \end{Proposition}
   
A key tool will be \cite[Corollary 2]{GS20},  which  showed that if $n\equiv a\pmod b$ and $n_{a,A}\leqslant n\leqslant bN-n_{b-a,b-A}$ then $n\in NA$ for all $N\geqslant 1$ if and only if
 $ N_{a,A} =\tfrac 1b (n_{a,A}+n_{b-a,b-A})$ for all $a$.

\begin{Case}
\label{Case first alpha equals 2}
If $A=\{ 0<a<2a<b\}$ with $(a,b) = 1$, then  \eqref{eq structure theorem}  holds for all $N \geqslant 1$. 
\end{Case} 
\begin{proof}Let $A':=b-A=\{ 0<2a'-b<a'<b\}$ with $a'=b-a$. For $1 \leqslant k  \leqslant (b-1)/2$ we have $n_{2ka,A}=2k\times a = k \times 2a$ while $n_{b-2ka,b-A}=n_{2ka',A'}=k\times (2a'-b)$; in the range $0 \leqslant k \leqslant (b-2)/2$, we have
  $n_{(2k+1)a,A}=a+k\times 2a$ while $n_{b-(2k+1)a,b-A}=n_{(2k+1)a',A'}=a'+k\times (2a'-b)$. So $N_{2ka,A} = k$ and $N_{(2k+1)a, A} = k + 1$, and so for all $1 \leqslant r \leqslant b-1$ we have $N_{r, A} = \frac{1}{b}(n_{r,A} + n_{b- r, b-A})$. Thus \cite[Corollary 2]{GS20} shows \eqref{eq structure theorem}  holds for all $N\geqslant 1$.
  \end{proof}
  
\begin{Case}
\label{Case second alpha equals 2}
If $A = \{0 < 2a - b < a < b\}$ with $(a,b) = 1$, then \eqref{eq structure theorem}  holds for all  $N \geqslant 1$. 
\end{Case}
\begin{proof} 
This follows from the previous case by symmetry. 
\end{proof}
 
\begin{Case}
\label{Case third alpha equals 2}
If $A=\{ 0<a,\tfrac b2<b\}$ with $(a,\tfrac b2)=1$, then \eqref{eq structure theorem}  holds for all  $N \geqslant 1$. 
\end{Case}
\begin{proof}
Here $b-A= \{ 0<b-a,\tfrac b2<b\}$ is of the same form. By performing a simple case analysis, we deduce that for $1 \leqslant k < \tfrac b2$ we have 
   $n_{ka,A}=k\times a$ and $n_{b-ka,b-A}=n_{k(b-a),b-A}=k\times (b-a)$, while for $0 \leqslant k < \tfrac b2$ we have
   $n_{ka+\tfrac b2,A}=k\times a+\tfrac b2$ 
   and $n_{\tfrac b2-ka,b-A}=n_{k(b-a)+\tfrac b2,A}=k\times (b-a)+\tfrac b2$. Then \eqref{eq structure theorem}  holds for all $N\geqslant 1$ by \cite[Corollary 2]{GS20}, as described above. 
   \end{proof}

 \begin{Case}
 \label{Case fourth alpha equals 2}
If $A = \{ 0<h<b-h<b\}$ with $(h,b) = 1$ then \eqref{eq structure theorem}  holds for all $N\geqslant b-1-h$. In particular  if $h \neq 1$ then \eqref{eq structure theorem} holds for all $N \geqslant b-3 = b - \ell - 1$. 
\end{Case}
\begin{proof}
If $a\not\equiv 0 \pmod b$ then the summands in $n_{a,A}$ are either all $h$ or all $b-h$ since if we had both we could remove one of each, contradicting minimality. Therefore
\begin{equation}
\label{eq: naA}
N_{kh,A} = \begin{cases} k   \\ b-k \end{cases} \text{ and }
n_{kh,A} = \begin{cases}
k h & \text{if } 1\leqslant k <b-h \\
(b-k)(b-h) & \text{if } b-h\leqslant k\leqslant b-1.
\end{cases} 
\end{equation}

If $k\leqslant h$ then $n_{kh,A}= kh$ and $n_{b-kh,b-A}=n_{k(b-h),A}=k\times (b-h)$. Then the structure \eqref{eq structure theorem}, restricted to the arithmetic progression $n \equiv kh \pmod b$, follows from \cite[Corollary 2]{GS20}.

If $h<k\leqslant \tfrac b2$ then $n_{kh,A} = kh$ and $n_{b-kh,b-A}=n_{(b-k)h,A}=(b-k)h$. Therefore we wish to show that if $n$ is in the range $kh\leqslant n\leqslant Nb-(b-k)h=kh+(N-h)b$ and $n\equiv kh \pmod b$ then $n\in NA$ (as long as $N \geqslant b-1-h$).

If we write $n=k\times h + j\times b$  for $j\in [0,N-k]$ then this covers such $n$ with
$kh\leqslant n\leqslant  kh+(N-k)b$; and if we write $n=(b-k)\times (b-h) + i\times b$ for $i\in[0,N+k-b]$ then we cover such $n$ with $kh+ (b-k-h)b\leqslant n\leqslant  kh+(N-h)b$. Together these two ranges cover the entire range of $n$, provided
$b-k-h\leqslant N-k+1$. This inequality holds, since $N\geqslant  b-1-h$.

To deal with the remaining arithmetic progressions $kh \pmod b$ for $k > \tfrac b2$ we note that $NA = bN - NA$, and so the result follows from the above using the arithmetic progression $-kh \pmod b$. 
\end{proof}

\begin{Case}
\label{Case fifth alpha equals 2}
If $A=\{ 0<a<a+\tfrac b2<b\}$ with $(a,\tfrac b2)=1$ then \eqref{eq structure theorem}  holds for all $N\geqslant  \tfrac b2$. 
\end{Case}
\begin{proof}
Note first that $b-A= \{ 0<\tfrac b2-a<b-a<b\}$, which is of the same form as $A$. The proof splits into four subcases, which we will deal with in two sets of two.

 If $1\leqslant k\leqslant \tfrac b2$ with $k$ even then 
$n_{ka,A}=k\times a$ and $n_{b-ka,b-A}=k\times (\tfrac b2-a)$. Therefore, from \eqref{eq aritmetic progressions},  we wish 
to represent all $n\equiv ka \pmod b$ with
$ka\leqslant n\leqslant ka+ b(N- \tfrac k2)$ by an element in $NA$.

If $1\leqslant k\leqslant \tfrac b2$ with $k$ odd then  
$n_{ka,A}=k\times a$ and $n_{b-ka,b-A}=(k-1)\times (\tfrac b2-a)+(b-a)$,
so wish to represent $n\equiv ka \pmod b$ with
$ka\leqslant n\leqslant ka+  b(N- \tfrac {k+1}2)$ by an element in $NA$. 

We let 
\[
n=(k-2i)\times a+2i\times (a+\tfrac b2)+ j\times b=ka+(i+j)b
\]
for $0\leqslant 2i\leqslant k$ and $0\leqslant j\leqslant N-k$. We have $n \in (k+j)A \subset NA$, and we obtain the full range of $n$ in each case, provided $N\geqslant k$. This is satisfied if $N \geqslant \frac{b}{2}$.

If $1\leqslant k< \tfrac b2$ with $k$ odd then 
$n_{ka+\tfrac b2,A}=(k-1)\times a+(a+\tfrac b2)$ and 
$n_{\tfrac b2-ka,b-A}=k\times (\tfrac b2-a)$,
so wish to represent $n\equiv ka+\tfrac b2 \pmod b$ with
$ka+\tfrac b2\leqslant n\leqslant ka+\tfrac b2+  b(N-\tfrac{k+1}2)$ by an element of $NA$.

If $1\leqslant k< \tfrac b2$ with $k$ even then 
$n_{ka+\tfrac b2,A}=(k-1)\times a+(a+\tfrac b2)$ and 
$n_{\tfrac b2-ka,b-A}=(k-1)\times (\tfrac b2-a)+(b-a)$,
so wish to represent $n\equiv ka+\tfrac b2 \pmod b$ with
$ka+\tfrac b2\leqslant n\leqslant ka+\tfrac b2+ b(N-1-\tfrac k2)$ by an element of $NA$. 
 
We let 
\[
n=(k-2i-1)\times a+(2i+1)\times (a+\tfrac b2)+ j\times b=ka+\tfrac b2+(i+j)b
\]
for $0\leqslant 2i+1\leqslant k$ and $0\leqslant j\leqslant N-k$. We have $n \in (k + j)A \subset NA$, and we obtain the full range provided $N\geqslant k$. This is satisfied if $N \geqslant \frac{b}{2}$. 
\end{proof}

%%%%%%%%%%%%%%%%%%%%%%
\subsection{Resolving the exceptional cases in which $H(2B) \neq \{0\}$}  
\begin{Case}
\label{Case first nontrivial Kneser}
If $A = \{0<a<\tfrac b2<a+\tfrac b2<b\}$ with $b$  even and $(a,\tfrac b2) = 1$ then \eqref{eq structure theorem}  holds  for $N \geqslant \frac{b}{2} - 1$.
\end{Case}
\begin{proof}
If $A^\prime: = \{0 < a  < a + b/2 < b\}$ then $\mathcal{P}(A) = \mathcal{P}(A^\prime) \cup \{n \geqslant 0: n \equiv b/2 \pmod b\}$, and $\mathcal{P}(b-A) = \mathcal{P}(b-A^\prime) \cup \{n \geqslant 0: n \equiv b/2 \pmod b\}$. From the proof of Case \ref{Case fifth alpha equals 2}, we see that \eqref{eq structure theorem} holds provided $N \geqslant \frac{b}{2} - 1$ except possibly for the residue class $n \equiv b/2 \pmod b$.  

However, since $n_{b/2,A} = n_{b/2,b-A} = b/2$, \cite[Corollary 2]{GS20} finishes the matter.
 \end{proof}

 \bibliographystyle{plain}
 \bibliography{frob}

\begin{thebibliography}{1}

\bibitem{GS20}
Andrew Granville and George Shakan.
\newblock The {F}robenius postage stamp problem, and beyond.
\newblock {\em Acta Math. Hungar.}, 161(2):700--718, 2020.

\bibitem{ILMZ12}
Geoffrey Iyer, Oleg Lazarev, Steven~J Miller, and Liyang Zhang.
\newblock Generalized more sums than differences sets.
\newblock {\em Journal of Number Theory}, 132(5):1054--1073, 2012.

\bibitem{Na72}
Melvin~B Nathanson.
\newblock Sums of finite sets of integers.
\newblock {\em The American Mathematical Monthly}, 79(9):1010--1012, 1972.

\bibitem{TV06}
Terence Tao and Van~H Vu.
\newblock {\em Additive combinatorics}, volume 105.
\newblock Cambridge University Press, 2006.

\bibitem{WCC11}
Jian-Dong Wu, Feng-Juan Chen, and Yong-Gao Chen.
\newblock On the structure of the sumsets.
\newblock {\em Discrete mathematics}, 311(6):408--412, 2011.

\end{thebibliography}

   \end{document}